\documentclass[12pt]{extarticle}
\usepackage{amsmath, amsthm, amssymb, mathtools, hyperref}
\usepackage[shortlabels]{enumitem}
\usepackage{graphicx}
\oddsidemargin \evensidemargin
\newtheorem{theorem}{Theorem}
\numberwithin{theorem}{section}
\newtheorem{proposition}[theorem]{Proposition}

\newtheorem{corollary}[theorem]{Corollary}

\newtheorem{remark}[theorem]{Remark}

\newtheorem{example}[theorem]{Example}

\newcommand{\RR}{\mathbb{R}}
\newcommand{\QQ}{\mathbb{Q}}
\newcommand{\PP}{\mathbb{P}}
\newcommand{\CC}{\mathbb{C}}
 
 \date{}
 
\title{\textbf{The Hurwitz Form of a Projective Variety}}
\author{Bernd Sturmfels}

\begin{document}
\maketitle

\begin{abstract} \noindent
The Hurwitz form of a variety is the discriminant that
characterizes linear spaces of complementary dimension
which intersect the variety in fewer than degree many points.
We study computational aspects of the Hurwitz form,
relate this to the dual variety and Chow form,
and show why
reduced degenerations are special on the Hurwitz polytope.
\end{abstract}

\section{Introduction}

Many problems in applied algebraic geometry can be
expressed as follows. We are given a fixed irreducible variety
$X $ in complex projective space $\PP^n$ of
dimension $d \geq 1$ and degree~$p$.
Suppose $X$ is defined over the field $\QQ$ of rational numbers.
We consider various linear subspaces $L \subset \PP^n$
of complementary dimension  $n-d$, usually defined 
over the real numbers in floating point representation.
The goal is to compute the intersection $L \cap X$ as accurately as possible.
If the subspace $L$ is generic then  $L \cap X$ consists of
$p$ distinct points with complex coordinates in $\PP^n$. 
How many of the $p$ points are real depends on the specific choice of~$L$.

In this paper we study
the discriminant associated with this family of polynomial systems.
The precise definition is as follows.
Let ${\rm Gr}(d,\PP^n)$ denote the Grassmannian of
codimension $d$ subspaces in $\PP^n$, and let
$\mathcal{H}_X \subset {\rm Gr}(d,\PP^n)$ 
be the subvariety  consisting of all subspaces $L$
such that $L \cap X$ does not consist of $p$ reduced points. 
The {\em sectional genus} of $X$, denoted $g = g(X)$, is the
arithmetic genus of the curve $X \cap L'$ where $L' \subset \PP^n$ is a general subspace
of codimension $d-1$. If $X$ is regular in codimension $1$
then the curve $X \cap L'$ is smooth (by Bertini's Theorem) and $g$ is its geometric genus.
The following result describes our object.

\begin{theorem} \label{thm:basic}
Let $X$ be an irreducible subvariety of $\PP^n$ having degree $p \geq 2$
and sectional genus $g$. Then $\mathcal{H}_X$ is an
irreducible hypersurface in the Grassmannian ${\rm Gr}(d,\PP^n)$, 
defined by an irreducible element $\,{\rm Hu}_X$ 
in the coordinate ring of ${\rm Gr}(d,\PP^n)$.
If  the singular locus of $X$ has codimension at least $2$ then
the degree of $ {\rm Hu}_X$ in Pl\"ucker coordinates equals
$2p+2g-2$.
\end{theorem}

The polynomial ${\rm Hu}_X$ defined here is the 
{\em Hurwitz form} of $X$. The name was chosen as a reference
to the Riemann-Hurwitz formula, which says that
a curve of degree $p$ and genus $g$ has $2p+2g-2$
ramification points when mapped onto $\PP^1$.
We say that ${\rm Hdeg}(X) := {\rm deg}({\rm Hu}_X) = 
2p+2g-2$ is the {\em Hurwitz degree} of $X$.
When $X$ is defined over $\mathbb{Q}$ then so is
$\mathcal{H}_X$.  Since $X$ is irreducible, its Hurwitz form ${\rm Hu}_X$ is irreducible.
With the geometric definition we have given, ${\rm Hu}_X$ is unique up to sign,
when written in Stiefel coordinates on ${\rm Gr}(d,\PP^n)$,
if we require it to have relatively prime integer coefficients.
When written in Pl\"ucker coordinates,  $\pm {\rm Hu}_X$ is unique 
only modulo the ideal of  quadratic Pl\"ucker relations.

The Hurwitz form ${\rm Hu}_X$ belongs to the family of
{\em higher associated hypersurfaces} described by
Gel'fand, Kapranov and Zelevinsky in \cite[Section 3.2.E]{GKZ}.
These hypersurfaces interpolate between the Chow form ${\rm Ch}_X$ 
and the $X$-discriminant. The latter is the equation of the dual variety $X^*$.
In that setting, the Hurwitz form ${\rm Hu}_X$ is only one step away
from the Chow form ${\rm Ch}_X$. An important result in \cite[Section 4.3.B]{GKZ} states
that these higher associated hypersurfaces are precisely the
coisotropic hypersurfaces in ${\rm Gr}(d,\PP^n)$, so their defining polynomials are
governed by the Cayley-Green-Morrison constraints for integrable distributions.

This article is organized as follows. In Section 2 we 
discuss examples, basic facts, 
and we derive Theorem~\ref{thm:basic}.
Section 3 concerns the Hurwitz polytope whose
vertices correspond to the initial Pl\"ucker monomials of 
the Hurwitz form. We compare this to the Chow polytope of \cite{KSZ}.
In Section 4 we define the Hurwitz form of a reduced cycle, and we show that this
 is compatible with flat families. As an application we resolve
 problems (4) and (5) in \cite[\S 7]{SSV}.

\section{Basics}

We begin with 
 examples that illustrate 
 Hurwitz forms in
computational algebraic geometry.

\begin{example}[Curves] \label{ex:curves}
\rm
If $X$ is a curve in $\PP^n$, so $d=1$, then
$\mathcal{H}_X = X^*$ is the hypersurface dual to $X$,
and ${\rm Hu}_X$ is the $X$-discriminant.
For instance, if $X$ is the rational normal curve in $\PP^n$ then
${\rm Hu}_X$ is the discriminant of a polynomial of degree $n$ in one variable.
For a curve $X$ in the plane $(n=2)$, the Hurwitz form is the polynomial 
that defines  the dual curve $X^*$, so the Hurwitz degree ${\rm Hdeg}(X)$ is the
degree  of $X^*$, which is $p(p-1)$ if $X$ is nonsingular. 
\hfill $\Diamond$
\end{example}

\begin{example}[Hypersurfaces]
\label{ex:hypersurfaces} \rm
Suppose that $X$ is a hypersurface in $\PP^n$, so $d=n-1$,
with defining polynomial $f(x_0,x_1,\ldots,x_n)$.
The Grassmannian ${\rm Gr}(n-1,\PP^n)$ is a manifold
of dimension $2n-2$ in the projective space $\PP^{\binom{n+1}{2}-1}$
with  dual Pl\"ucker coordinates $q_{01}, q_{02}, \ldots,q_{n-1,n}$.
We can compute
${\rm Hu}_X$ by first computing the discriminant of the univariate
polynomial function  $t \mapsto f( u_0 + t v_0, u_1 + t v_1 ,\ldots, u_n + t v_n)$, then
removing extraneous factors, and finally expressing the result in terms of 
$2 \times 2$-minors via $\,q_{ij} = u_i v_j - u_j v_i$.

We can make this explicit when $p=2$.
Let $M$ be a symmetric $(n+1) \times (n+1)$-matrix of rank $\geq 2$
and $X$ the corresponding quadric hypersurface in $\PP^n$.
We write $\wedge_2 M$ for the second exterior power
of $M$, and  $Q  = (q_{01}, q_{02}, \ldots,q_{n-1,n})$
for the row vector of dual Pl\"ucker coordinates.
With this notation, the Hurwitz form is the following quadratic form in the $q_{ij}$:
\begin{equation}
\label{eq:wedge2}  {\rm Hu}_X \,\, = \,\, Q \cdot (\wedge_2 M) \cdot Q^t.
\end{equation}

For a concrete example let $n = 3$ and consider the quadric surface
$X  = V(x_0x_3 - x_1 x_2)$.
Then ${\rm Hdeg}(X) = 2$ and the Hurwitz form equals
${\rm Hu}_X = q_{03}^2 + q_{12}^2 + 2 q_{03} q_{12} - 4 q_{02} q_{13}$.
When expressed in terms of Stiefel coordinates $u_i, v_j$, this is precisely
the {\em hyperdeterminant} of format $2 \times 2 \times 2$.
This is explained by Proposition \ref{prop:CayleyTrick}, with
$X = \PP^1 \times \PP^1$ and $Y = X \times \PP^1 \subset \PP^7$.
\hfill $\Diamond$
\end{example}

\begin{example}[Toric Varieties] 
\label{ex:toric}
\rm
Consider a toric variety $X_A \subset \PP^n$, defined by
a rank $d$ matrix $A \in \mathbb{Z}^{d \times (n+1)}$ with
$(1,1,\ldots,1)$ in its row space. Then
${\rm Hu}_{X_A}$ is the {\em mixed discriminant}
\cite{CCDRS} of $d$ Laurent polynomials in $d$ variables
with the same support $A$.
If $A$ is a unit square then
${\rm Hu}_{X_A}$ is the hyperdeterminant
seen above and in \cite[Example 2.3]{CCDRS}.
If $X_A$ is the $k$th Veronese embedding of $\PP^2$,
a surface of degree $p=k^2$ in $\PP^{\binom{k+2}{2}-1}$, 
then ${\rm Hu}_{X_A}$ is the classical
{\em tact invariant} which vanishes whenever two plane
curves of degree $k$ are tangent. Its degree is ${\rm Hdeg}(X_A) = 3 k^2 - 3k$.
An explicit formula for $k=2$ will be displayed in Example \ref{ex:quadvero}.
\hfill $\Diamond$
\end{example}

Our goal is to develop tools for writing
the Hurwitz form ${\rm Hu}_X$ explicitly as a polynomial.
There are four different coordinate systems for doing so,
depending on how the  subspace $L \in {\rm Gr}(d,\PP^n)$ is expressed.
If $L$ is the kernel of a $d \times (n {+}1)$-matrix then
the entries of that matrix are the {\em primal Stiefel coordinates}
and its maximal minors are the
{\em primal Pl\"ucker coordinates},  denoted $p_{i_1 i_2 \cdots i_d}$.
If $L$ is the row space of an $(n{+}1{-}d) \times (n{+}1)$-matrix then
the entries of that matrix are the {\em dual Stiefel coordinates}
and its maximal minors are the
{\em dual Pl\"ucker coordinates},  denoted $q_{j_0 j_1 \cdots j_{n-d}}$,
as in Example \ref{ex:hypersurfaces}.
In practice, one uses primal coordinates when
$d = {\rm dim}(X)$ is small, and one uses dual coordinates
when $n-d = {\rm codim}(X)$ is small.
The same conventions are customary for writing 
{\em Chow forms} \cite{DS, KSZ, Cortona}.

The hypersurface defined by the Chow form of $X$ is called the
{\em associated variety} in~\cite{GKZ, WZ}. In these references,
the Chow form is constructed as the
dual of a  Segre product. This is the
 {\em Cayley trick} of elimination theory.
 We now do the same for the Hurwitz form of~$X$.

\begin{proposition} \label{prop:CayleyTrick}
Let $X$ be an irrreducible variety of dimension $d$ and degree $p \geq 2$ in $\PP^n$,
and consider $Y = X \times \PP^{d-1}$ in its Segre embedding in
$\PP^{d(n+1)-1}$. The dual variety $Y^*$ is a hypersurface
in the dual $\,\PP^{d(n+1)-1}$. 
The Hurwitz form  $\,{\rm Hu}_X$  of
the given variety $X$, when written
in  the $d(n+1)$ primal Stiefel coordinates,
is equal to the defining polynomial of $Y^*$.
\end{proposition}

\begin{proof}
This is analogous to Theorem 2.7 in Section 3.2.D of \cite{GKZ}.
While that particular statement concerns  the case of the Chow form,
the discussion in \cite[Section 3.2.E]{GKZ} ensures that 
the result is the same for higher associated hypersurfaces,
such as the Hurwitz form.
\end{proof}

We next prove the statements about dimension and degree of
$\mathcal{H}_X$
 given in the introduction.
The idea is to reduce to the special case of curves, as discussed
in Example \ref{ex:curves}.

\begin{proof}[Proof of Theorem~\ref{thm:basic}]
Applying  Corollary 5.9 in \cite[Section 1.5.D]{GKZ}
to the representation in Proposition \ref{prop:CayleyTrick},
we find that $\mathcal{H}_X$ is a hypersurface
if and only if ${\rm codim}(X^*) \leq d$, or 
${\rm dim}(X^*) \geq n-d$. In light of
Corollary 1.2 in \cite[Section 1.1.A]{GKZ}, this happens
if and only if $X$ is not a linear space. So, 
since we assumed $p \geq 2$, this means that ${\mathcal H}_X$ is
a hypersurface.

Write $L = L' \cap H$ where $H$ is a varying hyperplane
and $L'$ is a fixed generic linear subspace of codimension $d-1$ in $\PP^{n}$.
The codimension $d$ subspace $L$ is a point in $\mathcal{H}_X$ if and only if
the zero-dimensional scheme
$X \cap L = X \cap (L' \cap H) = (X \cap L') \cap H$ is not reduced.
This happens if and only if $H$ is tangent to the curve $X \cap L'$
if and only if $H$ is a point in the projective variety dual to $X \cap L'$.
This curve is smooth and irreducible, by Bertini's Theorem, and it
has degree $p$ and genus $g$.
A classical result states
this dual to $X \cap L'$ is a  hypersurface of degree $2p+2g-2$;
see, for instance, the paragraph after Theorem 2.14 in \cite[Section 2.2.B]{GKZ}.
Hence
$\mathcal{H}_X$ is a hypersurface of that same degree in ${\rm Gr}(d,\PP^n)$.
\end{proof}

One motivation for studying the Hurwitz form ${\rm Hu}_X$ 
 comes from the analysis
of numerical algorithms for computing $L \cap X$. An appropriate tubular neighborhood
around $\mathcal{H}_X$ in ${\rm Gr}(d,\PP^n)$ is the locus
where the homotopy methods of numerical algebraic geometry
run into trouble. This is quantified by
the {\em condition number} of the algebraic function
$L \mapsto L \cap X$.  The quantity ${\rm Hdeg}(X) = 2p+2g-2$ is
crucial for bounding that condition number (cf.~\cite{BC}).

\begin{example} \label{ex:essential} \rm
This article was inspired by a specific application
to multiview geometry in computer vision, studied
in \cite[\S 3]{ALST}. The
 {\em variety of essential matrices} is a subvariety $X$ 
   in  the projective space $\PP^8$
of $3 \times 3$-matrices. Its real points are the rank $2$ matrices
whose two nonzero singular values coincide.
We have $d = 5$, $p = 10$, $g = 6$. This implies that
${\rm Hdeg}(X) = 30$, by Theorem  \ref{thm:basic}, so
${\mathcal H}_X$ is a hypersurface in the  Grassmannian ${\rm Gr}(5,\PP^8)$
whose defining polynomial ${\rm Hu}_X$ has degree $30$ in the
$\binom{9}{3} = 84$ Pl\"ucker coordinates.  \hfill $\Diamond$
\end{example}

\begin{remark} \rm
Let $h(m)$ be the Hilbert polynomial of $X \subset \PP^n$.
If $X$ is Cohen-Macaulay then we can read off
the Hurwitz degree from the 
$(d-1)$st derivative of the Hilbert polynomial:
$$ {\rm Hdeg}(X) \,\, = \,\, 2p+2g-2 \,\, = \,\, 
2 \cdot |h^{(d-1)}(-1)|. $$
Indeed, if $d=1$ then this comes from the
familiar  formula $\,h(m) = p m + (1-g) \,$
for the Hilbert polynomial of a projective curve.
For $d \geq 2$, the Cohen-Macaulay hypothesis
ensures that the numerator of the Hilbert series
remains the same under generic hyperplane sections,
and the Hilbert polynomial is transformed under such sections by
taking the derivative.

In {\tt Macaulay2} \cite{M2}, we can compute the integer
$\, |h^{(d-1)}(-1)| = {\rm Hdeg}(X)/2$ from the ideal {\tt I} of $X$
in two possible ways: take the
coefficient of ${\tt P}_{d-1}$ in
{\tt hilbertPolynomial(I)}, or add the last two entries in
{\tt genera(I)}. For instance, for the variety in Example \ref{ex:essential},
the former command gives
${\tt 6*P}_3 - {\tt 15*P}_4 + {\tt 10*P}_5$ and
the latter command gives
{\tt \{0,0,0,0,6,9\}}.
\end{remark}

If the parameters $d,n,p,g$ of the given variety $X$ are small enough, then
we can use computer algebra to determine the Hurwitz form ${\rm Hu}_X$,
and to write it as an explicit polynomial in the Pl\"ucker coordinates
on ${\rm Gr}(d,\PP^n)$. The following example will serve as an illustration:

\begin{example} \label{ex:quadvero} \rm
Let $X$ be  the Veronese surface in $\PP^5$, defined by the parametrization
$(x:y:z) \mapsto (x^2: xy:xz: y^2:yz:z^2)$. Here $n=5$, $d = 2$, $p = 4$, $g = 0$.
Following Example \ref{ex:toric}, the Hurwitz form
of $X$ is the classical tact invariant for $k=2$. We have the explicit formula
\begin{small}
$$ \begin{matrix} {\rm Hu}_X  \,= \,\,
 4 p_{01} p_{14}^2 p_{15}^3
- 4 p_{13} p_{12}^2 p_{15}^3
- 4 p_{34} p_{24}^2 p_{04}^3 
+ 4 p_{45} p_{14}^2 p_{04}^3
- 4 p_{25} p_{12}^2 p_{23}^3 
+ 4 p_{02} p_{24}^2 p_{23}^3  
 +  p_{12}^2 p_{14}^2  p_{24}^2 \\
\,+  \,p_{04}^2 p_{14}^2 p_{24}^2
+  p_{12}^2 p_{14}^2 p_{15}^2 
+  p_{12}^2 p_{23}^2 p_{24}^2
-16 p_{01} p_{13} p_{15}^4 
-16 p_{02} p_{25} p_{23}^4
-16 p_{34} p_{45} p_{04}^4
+ 256 p_{03}^2 p_{05}^2 p_{35}^2
\end{matrix}
$$
\end{small}
\vspace{-0.2cm}
\begin{tiny}
$$
\begin{matrix}
-12p_{04}^2p_{24}^3p_{13}-8p_{23}p_{04}^2p_{14}p_{24}^2+18p_{04}p_{24}^3p_{14}p_{03}
-8p_{23}^2p_{04}^2p_{24}^2+36p_{24}^3p_{03}p_{23}p_{04}-27p_{24}^4p_{03}^2 
-4p_{02}p_{14}^3p_{24}^2
-10p_{23}p_{14}p_{12}p_{24}^2p_{04} \\
+2p_{14}^2p_{12}p_{24}^2p_{04}  
  +12p_{02}p_{23}p_{14}^2p_{24}^2
+18p_{24}^3p_{14}p_{03}p_{12}-12p_{04} 
+p_{24}^3p_{12}p_{13} -18p_{23}p_{24}^3p_{03}p_{12}+20p_{23}^2p_{12}p_{24}^2p_{04}
-12p_{23}^2p_{02}p_{24}^2p_{14} \\ +2p_{23}p_{12}^2p_{24}^2p_{14}-4p_{12}^2p_{24}^3p_{13} 
-8p_{04}^3p_{14}p_{24}p_{35}+48p_{04}^3p_{24}p_{34}p_{15}+48p_{04}^3p_{14}p_{34}p_{25}
-80p_{04}^2p_{14}p_{24}p_{34}p_{05} +8p_{04}^3p_{24}p_{23}p_{35}  \\
+48p_{04}^3p_{23}p_{34}p_{25}  
{-} 120p_{04}^2p_{24}p_{03}p_{34}p_{25}+144p_{05}p_{24}^2p_{03}p_{04}p_{34}
-40p_{04}^2p_{24}p_{23}p_{05}p_{34}-12p_{14}^3p_{04}^2p_{25}+18p_{14}^3p_{04}p_{24}p_{05}  \\
-8p_{04}^2p_{14}^2p_{24}p_{15}   +28p_{04}^2p_{24}p_{14}p_{23}p_{15}+4p_{04}^2p_{24}p_{14}p_{25}p_{13} 
-52p_{04}p_{14}^2p_{24}p_{23}p_{05}-52p_{04}p_{14}p_{03}p_{24}^2p_{15}  
{+}120p_{05}p_{23}p_{14}p_{24}^2p_{03} \\
 +36p_{04}^2p_{24}^2p_{15}p_{13}+36p_{14}^2p_{23}p_{04}^2p_{25}
+72p_{24}^3p_{15}p_{03}^2+216p_{24}p_{15}p_{23}^2p_{04}^2-144p_{24}^3p_{13}p_{05}p_{03} 
{+}368p_{24}^2p_{13}p_{05}p_{23}p_{04}
-6p_{05}p_{14}^2p_{24}^2p_{03} \\ 
-264p_{04}^2p_{24}p_{25}p_{13}p_{23}+120p_{04}p_{24}^2p_{25}p_{13}p_{03} 
-264p_{04}p_{23}^2p_{14}p_{24}p_{05}-304p_{24}^2p_{15}p_{03}p_{23}p_{04}
+96p_{14}p_{23}^2p_{04}^2p_{25} +144p_{03}^2p_{23}p_{24}^2p_{25} \\ -16p_{23}^3p_{04}p_{24}p_{05} 
-24p_{05}p_{23}^2p_{24}^2p_{03}-160p_{03}p_{23}^2p_{24}p_{25}p_{04}
+48p_{23}^3p_{04}^2p_{25}-2p_{04}p_{14}^2p_{12}p_{24}p_{15} +12p_{15}p_{02}p_{14}^3p_{24}
+12p_{14}^3p_{12}p_{04}p_{25} \\
- 18p_{14}^3p_{12}p_{24}p_{05}+50p_{15}p_{14}p_{24}p_{12}p_{23}p_{04}
-46p_{14}p_{24}p_{12}p_{04}p_{25}p_{13} -46p_{14}^2p_{23}p_{12}p_{05}p_{24}
+8p_{14}^2p_{23}p_{12}p_{04}p_{25} -8p_{23}p_{15}p_{02}p_{14}^2p_{24} \\ -8p_{04}p_{12}p_{24}^2p_{15}p_{13}
-46p_{15}p_{14}p_{24}^2p_{03}p_{12}+92p_{14}p_{24}^2p_{12}p_{05}p_{13} 
+52p_{15}p_{24}p_{23}^2p_{12}p_{04} -58p_{15}p_{24}^2p_{23}p_{03}p_{12}
+6p_{03}p_{24}^2p_{12}p_{25}p_{13} \\ -36p_{23}^2p_{12}p_{14}p_{04}p_{25} 
-4p_{23}^2p_{12}p_{14}p_{05}p_{24}-36p_{23}^2p_{02}p_{15}p_{14}p_{24}
+52p_{13}p_{23}p_{24}p_{12}p_{25}p_{04}-40p_{23}^3p_{12}p_{05}p_{24} 
+48p_{23}^3p_{02}p_{15}p_{24}\\ -48p_{23}^3p_{12}p_{04}p_{25}+80p_{23}^2p_{24}p_{12}p_{25}p_{03}
-4p_{14}^3p_{12}^2p_{25}+2p_{14}^2p_{12}^2p_{15}p_{24} +18p_{12}^2p_{14}p_{24}p_{25}p_{13}
{-}12p_{12}^2p_{15}p_{24}^2p_{13} {-}12p_{23}p_{12}^2p_{14}^2p_{25} \\ +4p_{23}p_{12}^2p_{14}p_{15}p_{24} 
+18p_{23}p_{12}^2p_{24}p_{25}p_{13}-12p_{23}^2p_{12}^2p_{14}p_{25}
+2p_{23}^2p_{12}^2p_{15}p_{24}+128p_{04}^3p_{34}p_{05}p_{35} 
-128p_{04}^2p_{34}^2p_{05}^2+16p_{04}^4p_{35}^2 \\
+48p_{04}^3p_{15}^2p_{34}  -16p_{04}^3p_{14}p_{15}p_{35}+144p_{14}^2p_{04}p_{05}^2p_{34}
-24p_{14}^2p_{04}^2p_{05}p_{35}-160p_{04}^2p_{14}p_{15}p_{34}p_{05} 
-32p_{04}^3p_{25}p_{13}p_{35}-32p_{04}^3p_{15}p_{23}p_{35} \\
\end{matrix} $$ $$ \begin{matrix}
+416p_{04}p_{24}p_{15}p_{03}p_{34}p_{05} {+}128p_{04}^2p_{05}p_{14}p_{23}p_{35} 
{-}352p_{04}^2p_{25}p_{03}p_{34}p_{15}{+}160p_{04}^2p_{25}p_{13}p_{34}p_{05}
{-}384p_{04}^2p_{15}p_{23}p_{34}p_{05}{-}96p_{05}^2p_{13}p_{24}p_{04}p_{34} \\
+320p_{23}p_{14}p_{05}^2p_{04}p_{34}-192p_{05}^2p_{03}p_{24}p_{14}p_{34}
+384p_{04}p_{03}^2p_{25}^2p_{34}-288p_{03}^2p_{25}p_{24}p_{34}p_{05} 
+64p_{04}^2p_{25}p_{03}p_{23}p_{35}  -96p_{05}^2p_{03}p_{23}p_{24}p_{34} \\
-128p_{23}^2p_{04}^2p_{05}p_{35}-224p_{04}p_{25}p_{03}p_{23}p_{34}p_{05} 
+160p_{05}^2p_{23}^2p_{04}p_{34}
-27p_{14}^4p_{05}^2 
+36p_{14}^3p_{04}p_{05}p_{15}-8p_{04}^2p_{14}^2p_{15}^2
-72p_{14}^3p_{05}^2p_{23} \\ - 24p_{14}^2p_{13}p_{04}p_{05}p_{25} 
+96p_{04}^2p_{15}^2p_{24}p_{13}+144p_{05}^2p_{14}^2p_{24}p_{13}
-152p_{15}^2p_{04}^2p_{14}p_{23}+208p_{04}p_{14}^2p_{15}p_{23}p_{05} 
+104p_{15}^2p_{04}p_{14}p_{24}p_{03} \\ -24p_{14}^2p_{15}p_{05}p_{24}p_{03}
+104p_{04}^2p_{14}p_{15}p_{25}p_{13}-368p_{04}p_{14}p_{15}p_{05}p_{24}p_{13} 
+336p_{15}^2p_{03}p_{24}p_{23}p_{04}+336p_{04}p_{15}p_{23}^2p_{14}p_{05}  \\
+336p_{13}p_{05}p_{15}p_{24}^2p_{03} +336p_{04}p_{25}p_{13}^2p_{05}p_{24} 
-224p_{15}p_{04}p_{25}p_{13}p_{24}p_{03}-336p_{13}p_{05}p_{15}p_{24}p_{23}p_{04}
-112p_{23}p_{05}p_{15}p_{14}p_{24}p_{03} \\ -224p_{04}p_{25}p_{13}p_{05}p_{23}p_{14} 
+336p_{05}^2p_{24}p_{13}p_{23}p_{14}-56p_{04}^2p_{25}^2p_{13}^2
-336p_{05}^2p_{24}^2p_{13}^2-56p_{05}^2p_{14}^2p_{23}^2-56p_{15}^2p_{03}^2p_{24}^2 \\
-280p_{23}^2p_{15}^2p_{04}^232p_{14}p_{05}^2p_{23}^3
+320p_{23}p_{25}^2p_{03}p_{04}p_{13}+256p_{23}^2p_{15}p_{03}p_{25}p_{04} 
-128p_{23}p_{15}p_{03}^2p_{25}p_{24}  -288p_{23}^2p_{13}p_{04}p_{05}p_{25} \\
-64p_{23}^2p_{03}p_{05}p_{24}p_{15}+96p_{23}p_{25}p_{03}p_{05}p_{24}p_{13} 
-64p_{23}^2p_{05}^2p_{24}p_{13}  -192p_{03}^2p_{24}p_{25}^2p_{13}
+128p_{23}^3p_{03}p_{05}p_{25}+16p_{23}^4p_{05}^2-128p_{03}^2p_{23}^2p_{25}^2 \\
-12p_{14}^3p_{15}^2p_{02} {+} 18p_{14}^3p_{12}p_{05}p_{15} {-} 8p_{14}^2p_{15}^2p_{12}p_{04}
{-} 52p_{14}p_{15}p_{13}p_{12}p_{25}p_{04} {-} 36p_{14}^2p_{15}^2p_{02}p_{23} 
{+} 36p_{15}^2p_{04}p_{12}p_{24}p_{13} {-} 16p_{15}^2p_{14}p_{12}p_{23}p_{04} \\
+4p_{15}^2p_{14}p_{24}p_{03}p_{12}{-}6p_{14}^2p_{13}p_{12}p_{05}p_{25} 
{+}58p_{05}p_{12}p_{15}p_{14}^2p_{23}{+}80p_{12}p_{23}^2p_{05}p_{15}p_{14}
{+}40p_{23}p_{13}p_{15}p_{12}p_{25}p_{04}{+}120p_{24}p_{15}p_{03}p_{12}p_{25}p_{13}\\
+96p_{15}^2p_{23}^2p_{14}p_{02}+144p_{15}^2p_{23}p_{24}p_{03}p_{12}
-144p_{12}p_{13}^2p_{25}p_{05}p_{24}+24p_{23}p_{12}p_{13}p_{14}p_{05}p_{25} 
-160p_{15}^2p_{23}^2p_{12}p_{04}+72p_{12}p_{13}^2p_{25}^2p_{04} \\
-224p_{23}p_{12}p_{13}p_{15}p_{05}p_{24}
+56p_{05}p_{15}p_{23}^3p_{12} {+}120p_{12}p_{23}^2p_{13}p_{05}p_{25} 
-16p_{23}^2p_{15}p_{12}p_{25}p_{03}-144p_{23}p_{12}p_{25}^2p_{03}p_{13}-27p_{12}^2p_{13}^2p_{25}^2 \\
-48p_{02}p_{23}^3p_{15}^2  +  2 p_{15}^2p_{12}^2p_{14}p_{23} -12p_{12}^2p_{15}^2p_{24}p_{13} 
+18p_{12}^2p_{13}p_{14}p_{15}p_{25} +p_{15}^2p_{12}^2p_{23}^2+18p_{23}p_{12}^2p_{13}p_{15}p_{25}
-256p_{05}^2p_{04}p_{34}p_{03}p_{35} \\
+256p_{34}^2p_{03}p_{05}^3-128p_{05}p_{04}^2p_{03}p_{35}^2
-96p_{14}p_{13}p_{04}p_{05}^2p_{35}-32p_{15}^2p_{04}^2p_{03}p_{35} 
-288p_{05}^3p_{14}p_{13}p_{34}-96p_{05}^2p_{15}p_{03}p_{14}p_{34} \\
+64p_{05}^2p_{15}p_{04}p_{13}p_{34}+256p_{05}p_{15}p_{04}^2p_{13}p_{35} 
+256p_{15}^2p_{04}p_{03}p_{34}p_{05}-128p_{05}^3p_{23}p_{13}p_{34}
-128p_{05}p_{15}p_{03}p_{23}p_{35}p_{04}\\ -128p_{05}p_{25}p_{03}p_{13}p_{35}p_{04} 
+256p_{05}^2p_{25}p_{03}p_{13}p_{34}+256p_{05}^2p_{15}p_{03}p_{23}p_{34}
-128p_{03}^2p_{25}p_{15}p_{34}p_{05}+256p_{03}^2p_{25}p_{15}p_{35}p_{04} \\
-128p_{13}p_{23}p_{04}p_{05}^2p_{35} -256p_{03}^3p_{25}^2p_{35} 
+128p_{05}^2p_{03}p_{23}^2p_{35}+256p_{03}^2p_{23}p_{25}p_{05}p_{35} 
+48p_{15}^3p_{04}^2p_{13}-160p_{14}p_{15}^2p_{04}p_{13}p_{05} \\
-24p_{14}^2p_{15}^2p_{05}p_{03}-16p_{14}p_{15}^3p_{04}p_{03} +144p_{14}^2p_{13}p_{05}^2p_{15}  
-96p_{05}^2p_{13}^2p_{15}p_{24}
-192p_{13}^2p_{14}p_{05}^2p_{25}+64p_{15}^3p_{03}p_{23}p_{04} \\
+160p_{13}p_{15}^2p_{04}p_{23}p_{05}-352p_{13}p_{15}^2p_{04}p_{25}p_{03} 
-32p_{14}p_{13}p_{15}p_{05}^2p_{23}+160p_{15}^2p_{13}p_{05}p_{24}p_{03}
-128p_{15}^2p_{03}p_{05}p_{23}p_{14}\\
+416p_{05}p_{13}^2p_{15}p_{04}p_{25} 
{-}32p_{15}^3p_{03}^2p_{24}{+}384p_{03}^2p_{15}p_{25}^2p_{13}
{-}288p_{13}^2p_{05}p_{03}p_{25}^2 {-} 64p_{03}p_{15}^2p_{23}^2p_{05} 
{-} 64p_{03}^2p_{15}^2p_{23}p_{25} {-} 64p_{05}^2p_{13}p_{15}p_{23}^2 \\
{+} 384p_{13}^2p_{23}p_{05}^2p_{25}{-}256p_{05}p_{13}p_{15}p_{23}p_{25}p_{03} 
-80p_{14}p_{13}p_{12}p_{15}^2p_{05}+48p_{13}p_{14}p_{15}^3p_{02}
+48p_{15}^3p_{13}p_{12}p_{04}-8p_{14}p_{15}^3p_{12}p_{03} \\
-48p_{13}p_{15}^3p_{02}p_{23}-8p_{15}^3p_{03}p_{12}p_{23} 
-120p_{15}^2p_{13}p_{12}p_{25}p_{03}+16p_{12}p_{13}p_{15}^2p_{05}p_{23} 
+144p_{13}^2p_{12}p_{15}p_{05}p_{25} \\
+256p_{05}^3p_{13}^2p_{35}-256p_{13}p_{03}p_{05}^2p_{15}p_{35} 
-128p_{03}^2p_{15}^2p_{05}p_{35}
+128p_{15}^3p_{03}p_{13}p_{05}+16p_{15}^4p_{03}^2-128p_{05}^2p_{13}^2p_{15}^2
\end{matrix}
$$
\end{tiny}

The first $14$ monomials are the initial terms
with respect to the torus action on ${\rm Gr}(2,\PP^5)$.
In addition to these $14$ weight components, there are
$45$ other weight components from the non-initial terms.
Substituting $p_{ij} = a_i b_j - a_j b_i $
in ${\rm Hu}_X$ gives a  sum of
$3210$ monomials of bidegree $(6,6)$. This is
 the tact invariant (cf.~Example \ref{ex:toric})
of two ternary quadrics.
 \hfill $\Diamond$
\end{example}

Clearly, as the parameters $d,p,n,g$ increase, it soon becomes
prohibitive to compute such an explicit expansion for ${\rm Hu}_X$.
We can still hope to compute the initial terms from various
initial monomial ideals of $X$. This will be studied in the next section.
Another possibility is to use {\em numerical algebraic geometry}
to compute and represent the Hurwitz form ${\rm Hu}_X$.
We illustrate  the capabilities of {\tt Bertini} \cite{BHSW} with
a non-trivial example from computer vision.

\begin{example} \rm
The {\em calibrated trifocal variety}
$X$ has dimension $d = 11$ and it  lives in $\PP^{26}$. 
 This extends the variety of essential
 matrices in Example \ref{ex:essential}
 from two cameras to three cameras.
 It is obtained from the parametrization of the trifocal variety in
 \cite{AO} by specializing each of the three
 camera matrices to have a
  rotation matrix in its left $3 \times 3$-block.
 Computations in {\tt Bertini} \cite{BHSW} carried out by
 Joe Kileel and Jon Hauenstein revealed
that $p = 4912$ and $g = 13569$ respectively.
The details of this computation, and the equations  known
to vanish on $X$, will be described elsewhere.
 Hence for this variety, we have   ${\rm Hdeg}(X) = 36960$.  \hfill $\Diamond$
\end{example}

\section{The Hurwitz polytope}

An important combinatorial invariant of hypersurface in
a projective space $\PP^n$ is its Newton polytope. 
This generalizes to hypersurfaces in a  Grassmannian ${\rm Gr}(d,\PP^n)$
if we take the weight polytope with respect to the action of the
$(n+1)$-dimensional torus on ${\rm Gr}(d,\PP^n)$.
This torus action corresponds to
the $\mathbb{Z}^{n+1}$-grading on
the primal Pl\"ucker coordinates by
$$ {\rm deg}(p_{i_1 i_2 \cdots i_d}) \,\, = \,\, {\bf e}_{i_1} + 
{\bf e}_{i_2} + \cdots + {\bf e}_{i_d}. $$
The weight polytope of the Chow form of a variety
$X \subset \PP^n$ was studied in \cite{KSZ}.
It is known as the {\em Chow polytope}.
In this section we study the analogous polytope for the Hurwitz form.

We define the {\em Hurwitz polytope}, denoted ${\rm HuPo}_X$, of a projective variety 
$X \subset \PP^n$ to be the weight polytope of the Hurwitz form 
${\rm Hu}_X$.  This definition is to be understood as follows.
The coordinate ring of ${\rm Gr}(d,\PP^n)$
is $\mathbb{Z}^{n+1}$-graded. The Hurwitz form ${\rm Hu}_X$ is an element
of that ring, so it decomposes uniquely into 
$\mathbb{Z}^{n+1}$-graded components.
The Hurwitz polytope
${\rm HuPo}_X$ is the convex hull in $\RR^{n+1}$ of all
degrees that appear in this decomposition.

\begin{example} \rm
For the Segre variety $X = \PP^1 \times \PP^1$ in $\PP^3$, 
the Hurwitz form
is invariant under passing from dual Pl\"ucker coordinates
(as in Example \ref{ex:hypersurfaces})
to primal Pl\"ucker coordinates. The weights occurring in
${\rm Hu}_X = p_{03}^2 + p_{12}^2 + (2 p_{03} p_{12} - 4 p_{02} p_{13})$
are the following three points:
 $$ \hbox{
 ${\rm degree}(p_{03}^2) = (2,0,0,2)$,
 ${\rm degree}(p_{12}^2) =  (0,2,2,0)\,\,$ 
 and $\,\,{\rm degree}(p_{02}p_{13}) = (1,1,1,1)$.}
 $$
Their convex hull in $\RR^4$ is a line segment, and this is
the Hurwitz polytope ${\rm HuPo}_X$.
We note that the Chow form of our surface equals ${\rm Ch}_X = p_0 p_3 - p_1 p_2 $,
in dual Pl\"ucker coordinates. The line segment
${\rm HuPo}_X$ is twice the line segment
${\rm ChPo}_X = {\rm conv} \{ (1,0,0,1), (0,1,1,0) \}$.
The latter is the Chow polytope of $X$.
  \hfill $\Diamond$
\end{example}

Integer vectors ${\bf w} \in \mathbb{Z}^{n+1}$ represent one-parameter subgroups
of the torus action on $\PP^n$. Their action on subvarieties $X$ is compatible with the construction of
Hurwitz forms:
\begin{equation} \label{eq:epsilonHu}
{\rm Hu}_{ \epsilon^{\bf w} X}(p) \,\,\,  = \,\,\, {\rm Hu}_X(\epsilon^{-{\bf w}} p) .
\end{equation}
Here $\epsilon$ is a parameter and $\epsilon^{\bf w} = 
(\epsilon^{w_0},\epsilon^{w_1},\ldots,\epsilon^{w_n})$, and the action is
the same as that on Chow forms seen in
\cite[\S 2.3]{DS}. For $\epsilon \rightarrow 0$, the polynomial
in (\ref{eq:epsilonHu}) is the {\em initial form} ${\rm in}_{\bf w}({\rm Hu}_X)$.
For generic ${\bf w}$, this initial form is fixed under
the action of the torus $(\CC^*)^{n+1}$ on ${\rm Gr}(d,\PP^n)$, and 
it can hence be expressed as a monomial in the Pl\"ucker coordinates.
The following example illustrates these
{\em initial monomials} and how they determine
the Hurwitz polytope.

\begin{example} 
\label{ex:VeroHurPo}
\rm
Let $X$ be the Veronese surface in $\PP^5$. Its Hurwitz form ${\rm Hu}_X$
was displayed explicitly in Example \ref{ex:quadvero}.
  The Hurwitz polytope 
${\rm HuPo}_X$ is $3$-dimensional and it has $14$ vertices,
namely the weights of the first $14$  Pl\"ucker monomials, which were shown in
larger font. The weights of these monomials, in the given order, are the
  following $14$ points in $\RR^6$:
$$ \begin{matrix} 
(1 6 0 0 2 3),  (0 6 2 1 0 3),(3 0 2 1 6 0), (3 2 0 0 6 1),  (0 2 6 3 0 1), (1 0 6 3 2 0),  (0 4 4 0 4 0),  \\
   (2 2 2 0 6 0) ,  (0 6 2 0 2 2),    (0 2 6 2 2 0), 
 (1 6 0 1 0 4) , (1 0 6 4 0 1),   (4 0 0 1 6 1), (4 0 0 4 0 4).
 \end{matrix}
$$
The Chow form and Chow polytope of the Veronese surface $X$ 
 were displayed explicitly in  \cite[Section 2, page 270]{Cortona}. This
 should be compared to the Hurwitz form and Hurwitz polytope.
The initial monomials of  ${\rm Ch}_X$
correspond to the $14$ maximal toric degenerations of $X$, 
one for each of the $14$ triangulations of the
triangle $2 \Delta =  {\rm conv} \{(200),(020),(002)\}$.
 The Chow polytope ${\rm ChPo}_X$
is {\em normally equivalent} to the Hurwitz polytope ${\rm HuPo}_X$.
This means that both polytopes share the same normal fan in $\RR^6$.
In particular, they have the same combinatorial type.
That type is the {\em $3$-dimensional  associahedron}.
The correspondence between the initial monomials of ${\rm Ch}_X$
and those of ${\rm Hu}_X$ is  as follows, up to symmetry:
$$
\begin{matrix}
{\rm Ch}_X & &
p_{012} p_{124} p_{134} p_{345} &\,
p_{014} p_{024} p_{134} p_{245}  &\,
p_{015}^2 p_{134} p_{145} &\,
p_{015}^2 p_{135}^2 & \,
p_{035}^4  \\
{\rm  Hu}_X & &
    p_{12}^2 p_{14}^2  p_{24}^2 &
  p_{04}^2 p_{14}^2 p_{24}^2 &
    p_{01} p_{14}^2 p_{15}^3 &
 p_{01} p_{13} p_{15}^4  &\,
  p_{03}^2 p_{05}^2 p_{35}^2
\end{matrix}
$$
The first two  degenerations are reduced unions
of four coordinate planes in $\PP^5$, corresponding to the
unimodular triangulations of $2\Delta$.
The last degeneration is a plane of multiplicity four, corresponding to the trivial triangulation
of $2\Delta$ into a single large triangle $ \{0,3,5\}$.
 \hfill $\Diamond$
\end{example}

Example \ref{ex:VeroHurPo}
 raises the question whether
the vertices of the Hurwitz polytope and 
the Chow polytope are always in bijection. The answer  is ``no''
as the following example shows.

\begin{example}[Plane conics] \rm
Let $X$ be a plane conic in $\PP^2$ defined by the quadratic form
\begin{equation}
\label{eq:primalquadric}
 \begin{pmatrix} x_0 &  x_1 & x_2 \end{pmatrix} 
\begin{pmatrix}
m_{00} & m_{01} & m_{02} \\
m_{01} & m_{11} & m_{12} \\
m_{02} & m_{12} & m_{22} 
\end{pmatrix}
\begin{pmatrix} x_0 \\ x_1 \\ x_2 \end{pmatrix},
\end{equation}
where the $m_{ij}$ are scalars in $\mathbb{Q}$.
This equation agrees with the Chow form of $X$, so generally
the Chow polytope is the triangle
$ {\rm Ch}_X \,= \,2 \Delta$.
According to Example \ref{ex:curves}, the hypersurface $\mathcal{H}_X$ is the
dual conic $X^*$.
Hence the Hurwitz form is the quadratic form of the adjoint
\begin{equation}
{\rm Hu}_X \,=\, 
\label{eq:dualquadric}
 \begin{pmatrix} p_0 \! & \!  p_1 \! &\! p_2 \end{pmatrix} \!\!
\begin{pmatrix}
m_{11} m_{22} - m_{12}^2 &
 m_{12} m_{02} {-}m_{01} m_{22} &
m_{01} m_{12} {-} m_{11} m_{02} \\
 m_{12} m_{02} {-} m_{01} m_{22} &
m_{00} m_{22} - m_{02}^2 &
  m_{01} m_{02} {-} m_{00} m_{12} \\
m_{01} m_{12} {-} m_{11}  m_{02} &
 m_{01} m_{02} {-} m_{00} m_{12}  &
m_{00}  m_{11} - m_{01}^2
\end{pmatrix}\!\!
\begin{pmatrix} p_0 \\ p_1 \\ p_2 \end{pmatrix}\!.
\end{equation}
This agrees with  formula (\ref{eq:wedge2}) if we set
  $p_0 = q_{12}, p_1 = -q_{02} $ and $p_2 = q_{01}$.
  The Hurwitz polytope $ {\rm HuPo}_X $ is
  a subpolytope of the triangle $2 \Delta= {\rm conv} \{(200),(020),(002)\}$.
There are several combinatorial possibilities, depending on which
matrix entries in (\ref{eq:dualquadric}) are zero. For instance,
if $m_{11} m_{22} = m_{12}^2$ but the $m_{ij}$ are otherwise generic
then ${\rm HuPo}_X$ is a quadrilateral.
Similarly if $m_{11} = 0$ then ${\rm ChPo}_X $ is quadrilateral
and ${\rm HuPo}_X$ is a triangle. Hence, there is generally no map
from the set of vertices of one polytope to the  vertices of the other.
  \hfill $\Diamond$
\end{example}

A geometric explanation for  this example
is given by Katz' analysis in \cite{Kat} of
the duality of plane curve under degenerations.
Suppose that $X_\epsilon$ is a plane curve defined by
a ternary form
$\, f^q + \epsilon g \,$
where $f$ and $g$ are general ternary forms of degree $r$
and $p = qr$ respectively.
Here $\epsilon $ is a parameter.
For $\epsilon \not= 0$, the curve
$X_\epsilon$ is smooth, and the Hurwitz form
${\rm Hu}_{X_\epsilon}$ defines
the dual curve, of degree $p(p-1) = q^2r^2 - q r$.
Now consider the limit $\epsilon \rightarrow 0$.
The limit curve $X_0$ is $V(f)$ with multiplicity $q$,
with no trace of $V(g)$,
but the limit of the Hurwitz form remembers the
points in the intersection $V(f,g)$.

\begin{proposition} \label{prop:katz} {\rm \cite[Proposition 1.2]{Kat}} \ \ 
The constant term of ${\rm Hu}_\epsilon$ with respect to $\epsilon$ equals
\begin{equation}
\label{eq:katzlimit}
 {\rm Hu}_{X_\epsilon}|_{\epsilon = 0}  \quad = \quad
({\rm Hu}_{V(f)})^q \cdot \prod_{u \in V(f,g)} ({\rm Ch}_u)^{q-1} .
\end{equation}
\end{proposition}

Note that the variety $V(f,g)$ consists of $rp = qr^2$ points
$u = (u_0:u_1:u_2)$. For each of these, the linear form
${\rm Ch}_u \, = \,u_0 p_0 + u_1 p_1 + u_2 p_2$ appears as a factor.
The Hurwitz form ${\rm Hu}_{V(f)}$ of the curve  $V(f)$ has
degree $r(r-1)$. Hence the right hand side in
(\ref{eq:katzlimit}) has degree
$$ r(r-1)q  \,+\,    (qr^2) (q-1) \quad = \quad p(p-1). $$
   
Proposition  \ref{prop:katz} is indicative of 
what happens to the limit of the Hurwitz form 
${\rm Hu}_{X_\epsilon}$ when a family of
irreducible varieties $X_\epsilon$ degenerates to a non-reduced cycle
or scheme $X_0$. The limit $ {\rm Hu}_{X_\epsilon}|_{\epsilon = 0}$  remembers
information about the family that cannot be recovered from~$X_0$.

\section{Reduced degenerations}

In this section we consider families $X_\epsilon $
whose limit object $X_0$ is a reduced cycle,
and we show that the indeterminacy seen in Proposition \ref{prop:katz}
no longer happens. Instead, we will have
\begin{equation}
\label{eq:HuDegen}
 {\rm Hu}_{X_\epsilon}|_{\epsilon = 0} \,\, = \,\, {\rm Hu}_{X_0} . 
 \end{equation}

To make sense of this identity, we now define the Hurwitz form of
a {\em reduced cycle} $Y$ in $\PP^n$.
Let  $Y = \sum_{i=1}^l Y_i$ where
$Y_1,\ldots,Y_l$ are distinct irreducible $d$-dimensional subvarieties in $\PP^n$.
Let $Z_1, \ldots, Z_m$ be the distinct irreducible
varieties of dimension $d-1$ that arise as components in the
pairwise intersections
 $Y_i \cap Y_k$. We write $\nu_j$ for the multiplicity of
  the one-dimensional local ring $\mathcal{O}_{Y,Z_j}$ at its maximal ideal.
If all pairwise intersections $Y_i \cap Y_k$ are transverse then
$\nu_j$ simply counts the number of components $Y_i$ of $Y$ that contain $Z_j$.

We now define the {\em Hurwitz form} of the reduced cycle $Y$ as follows:
\begin{equation}
\label{eq:HuCycle}
 {\rm Hu}_Y \,\, = \,\,\, \prod_{i=1}^l {\rm Hu}_{Y_i}  \cdot \prod_{j=1}^m ({\rm Ch}_{Z_j})^{2 \nu_j-2} .
 \end{equation}
Here ${\rm Hu}_{Y_i}$ is the Hurwitz form of a $d$-dimensional variety,
while ${\rm Ch}_{Z_j}$ is the Chow form of a $(d-1)$-dimensional variety,
so they both define hypersurfaces in the same Grassmannian ${\rm Gr}(d,\PP^n)$.
Our main result in this section states that this is the correct definition for limits.

\begin{theorem}
\label{thm:flat}
Let $(X_\epsilon)$ be a flat family of subschemes in $\PP^n$,
where the general fiber (for $\epsilon \not=0 $) is irreducible
and regular in codimension $1$, 
and the special fiber (for $\epsilon = 0$) is reduced and  each of its
irreducible components is regular in codimension $1$.
The
Hurwitz form of $X_\epsilon$ satisfies the 
identity (\ref{eq:HuDegen}) with the Hurwitz form
${\rm Hu}_{X_0} $ of the special fiber defined as in (\ref{eq:HuCycle}).
\end{theorem}

\begin{proof}
By intersecting with a general linear subspace $L'$ of codimension $d-1$ in $\PP^n$, as in the proof
of Theorem \ref{thm:basic}, we reduce the statement to
the $d=1$ case, when $(X_\epsilon)$ is a family of curves.
In that case, the Hurwitz form ${\rm Hu}_{X_\epsilon}$ is the 
polynomial that defines the hypersurface $(X_\epsilon)^*$
dual to the curve $X_\epsilon$.  Here we are tacitly assuming that
$X$ is a curve of degree $\geq 2$.

We first consider the case of planar curves $(n=2)$.
Here, our result follows
from  the {\em General Class Formula} in \cite[Section A.5.4]{Fis}.
Indeed, the $Z_i$ correspond to singular points on the curve,
and the $\nu_i$ are the degrees of the singularities 
in the sense of \cite[Section A.5.2]{Fis}.

For the general case $(n \geq 3)$, we consider a
random projection $\pi : \PP^n \dashrightarrow \PP^2$,
and we examine the degeneration of the 
irreducible plane curve $\pi(X_\epsilon)$ to its limit $\pi(X_0)$.
The equation of the dual curve $\pi(X_\epsilon)^*$ is an 
irreducible polynomial.  In particular, 
the singularities of $\pi(X_\epsilon)$ that were acquired by the
projection do not appear as factors.
By \cite[Proposition I.1.4.1, page 31]{GKZ},  the equation of $\pi(X_\epsilon)^*$
is found by restricting the irreducible polynomial ${\rm Hu}_{X_\epsilon}$ to
a general plane $\PP^2$ in $\PP^n$, namely the plane dual to the base locus of $\pi$.
As $\epsilon$ approaches $0$, we obtain the restriction of ${\rm Hu}_{X_0} $
to the same plane. This now factors as in the previous paragraph, with one reduced factor for each
irreducible component of $X_0$ or $\pi(X_0)$, and non-reduced factors 
for the codimension $1$ intersections of these components.
\end{proof}

\begin{remark} \rm
Theorem \ref{thm:flat} requires the
hypothesis that the components of $X_0$ are
regular in codimension $1$. For instance,
suppose that $(X_\epsilon)$ is a family of
smooth cubic curves in $\mathbb{P}^2$ and
$X_0$ is a nodal cubic. Then $H_{X_0}$ has degree $4$ but
the left hand side of (\ref{eq:HuDegen})  has degree $6$.
\end{remark}

\begin{remark} \rm
Formula (\ref{eq:HuCycle}) can be viewed a generalization of
the following familiar equation relating the resultant of two 
polynomials in one variable to the discriminant of their product:
$$ {\rm Discrim}(f_1 \cdot f_2) \,\,= \,\, {\rm Discrim}(f_1) \cdot {\rm Discrim}(f_2) \cdot {\rm Res}(f_1, f_2)^2. $$
For an extension to linear
sections of toric varieties see \cite[Corollary 6]{DEK}.
\end{remark}

An important special case of Theorem \ref{thm:flat} arises when
the limit cycle $Y = X_0$ is an arrangement of linear subspaces of dimension $d$
in $\PP^n$.
Here, the first factor on the right hand side in (\ref{eq:HuCycle}) disappears, and 
the Hurwitz form of the subspace arrangement $Y$ equals
\begin{equation}
\label{eq:HuCycle2}
 {\rm Hu}_Y \,\, = \,\,\, \prod_Z ({\rm Ch}_{Z})^{2 \nu(Z)-2} .
 \end{equation}
 In this formula,
  $Z$ runs over all strata of dimension $d-1$ 
 in the subspace arrangement  $Y$, and $\nu(Z)$ is the number of $d$-planes
   $Y_i$ that contain  $Z$.
   
Of particular interest is the scenario when $Y$ consists of
coordinate subspaces in $\PP^n$. Here, $Y$
can be identified with a simplicial complex of dimension $d$
on the vertex set $\{0,1,\ldots,n\}$. The product 
(\ref{eq:HuCycle2}) is over all $(d-1)$-simplices $Z$.
These correspond to coordinate planes
$$ {\rm span}({\bf e}_{i_1},{\bf e}_{i_2},\ldots,{\bf e}_{i_d}) \,\, = \,\,
V(x_{j_0}, x_{j_1},\ldots,x_{j_{n-d}}). $$
These are indexed by set partitions 
$\{0,1,\ldots,n\} = \{i_1,i_2,\ldots,i_d \} \cup \{j_0,j_1,\ldots,j_{n-d}\}$,
and their Chow forms are just Pl\"ucker variables,
as in \cite[Proposition 3.4]{Cortona}.

\begin{equation}
\label{eq:HuCycle3}
 p_Z \,\, := \,\, p_{i_1 j_2 \cdots j_d} \,\, = \,\, q_{j_0 j_1 \cdots j_{n-d}} \,\,= \,\,
{\rm Ch}_Z. 
\end{equation}
This situation arises whenever the ideal of a projective variety has a 
Gr\"obner basis with a squarefree initial ideal. This is now the
Stanley-Reisner ideal of the simplicial complex $Y$.

\begin{corollary}
\label{cor:leadlead}
Let $I$ be a homogeneous prime ideal in $K[x_0,x_1,\ldots,x_n]$ and
suppose  $M = {\rm in}_{\bf w}(I)$ is a squarefree initial monomial ideal.
The initial form ${\rm in}_{\bf w}({\rm Hu}_{V(I)})$
of the Hurwitz form 
equals the monomial $\,\prod_Z p_Z^{2 \nu(Z)-2}$
where $Z$ runs over codimension $1$ simplices in $V(M)$.
\end{corollary}

\begin{proof}
Using (\ref{eq:HuCycle2}) and the identification (\ref{eq:HuCycle3}),
 this is a direct consequence of Theorem~\ref{thm:flat}.
\end{proof}

This corollary allows us to compute the Hurwitz polytope and initial terms of the Hurwitz form
for any homogeneous ideal whose initial monomial ideals are all squarefree.
One situation where this holds is the ideal generated by the maximal minors
of a rectangular matrix of unknowns, by \cite{BZ, SZ}. Here the initial ideals
correspond to the {\em coherent matching fields} in \cite{SZ}, and each of these determines
an initial Pl\"ucker monomial in the Hurwitz form for the determinantal
variety of maximal minors. We illustrate this with an example.

\begin{example}[Ideals of Maximal Minors] \rm
The  four $3 \times 3$-minors of the $3 \times 4$-matrix
$$ \begin{pmatrix}
x_0 & x_1 & x_2 & x_3 \\
x_4 & x_5 & x_6 & x_7 \\
x_8 & x_9 & x_{10} & x_{11}
\end{pmatrix}
$$
form a universal Gr\"obner basis for the ideal $I$ they generate \cite[Theorem 7.2]{SZ}.
The  variety $X = V(I) \subset \PP^{11}$  represents the $3 \times 4$-matrices of rank $\leq 2$.
Here, $n= 11$, $d = 9$, $p = 6$, and $g = 3$, so the Hurwitz degree equals
${\rm Hdeg}(X) = 16$.  The determinantal ideal $I$
has $96$ initial monomial ideals, all squarefree, in two symmetry classes.
Using Corollary \ref{cor:leadlead}, we 
read off the corresponding initial Pl\"ucker monomials of degree $16$ in 
the Hurwitz form ${\rm Hu}_X$:
$$ \begin{matrix}
 & \hbox{initial monomial ideal} & &\hbox{initial Pl\"ucker monomial in ${\rm Hu}_X$} \\
72 \,\,{\rm ideals}\,\,{\rm like} & \langle x_2 x_5 x_8, x_3 x_5 x_8, x_3 x_6 x_8 ,x_3 x_6 x_9 \rangle &  
\quad \leadsto &
q_{235}^2 q_{238}^2 q_{358}^2 q_{356}^2 
q_{368}^2 q_{389}^2 q_{568}^2 q_{689}^2  \\
 24 \,\,{\rm ideals}\,\,{\rm like} & \langle x_0 x_6 x_9, x_3 x_4 x_9 , x_3 x_6 x_8, x_3 x_6 x_9 \rangle  & 
\quad \leadsto &
q_{036}^2 q_{039}^2 q_{346}^2 
q_{389}^2 q_{469}^2 q_{689}^2 q_{369}^4 
\\
\end{matrix}
$$
So, while the Hurwitz form ${\rm Hu}_X$ is a huge polynomial that is hard to compute explicitly,
it is easy to write down the $96$ initial Pl\"ucker monomials of ${\rm Hu}_X$,
one for each vertex of the Hurwitz polytope ${\rm HuPo}_X$.
This polytope has dimension $6$, it is simple, and has the same
normal fan as the Chow polytope
${\rm ChPo}_X$. By \cite[Theorem 2.8]{SZ}, this is the
 {\em transportation polytope} of 
nonnegative $3 \times 4$-matrices whose rows sum to $4$
and whose columns sum to $3$.
  \hfill $\Diamond$
\end{example}

Here is 
another important class of varieties
all of whose initial ideals are squarefree.

\begin{example}[Reciprocal Linear Spaces] 
\label{ex:RLS}
\rm
Fix a $(d+1) \times (n+1)$-matrix $A$ of rank $d+1$ with entries in $\mathbb{Q}$,
and let $X$ be the reciprocal of the row space of $A$. Thus $X$ is the
Zariski closure in $\PP^n$ of the set of points $(u_0:u_1:\cdots:u_n)$ 
with all coordinates nonzero and such
that $(u_0^{-1},u_1^{-1},\ldots,u_n^{-1}) \in {\rm rowspace}(A)$.
Proudfoot and Speyer \cite{PS} show that the circuits of $A$
define a universal Gr\"obner basis for the ideal of~$X$,
and each initial monomial ideal corresponds to 
the {\em broken circuit complex} of $A$
under some ordering of $\{0,1,\ldots,n\}$.
For example, if $A$ is generic, in the sense
that all maximal minors of $A$ are nonzero, then
$p = \binom{n}{d}$, and the $p$ facets of the broken circuit complex are
 $\{0,i_1,\ldots,i_d\}$ where $1 \leq i_1< \cdots < i_d \leq n$.
 The corresponding initial monomial of the Hurwitz form~is
 $$ \prod (p_{0 i_2 \cdots i_d})^{2 (n-d)} , $$
where the product is over all $(d-1)$-element subsets 
$\{i_2,\ldots,i_d\}$ of $\{1,\ldots,n\}$.
Therefore,
$$ {\rm Hdeg}(X) \,\,\,= \,\,\, 2 \binom{n}{d-1}(n-d). $$
If $A$ is not generic then ${\rm Hdeg}(X)$ is given by
a matroid invariant that appears in \cite{DSV} and \cite{SSV}.

Our study of the Hurwitz form 
 furnishes the answer to Question 4
 in \cite[\S 7, p.~706]{SSV}:
 {\em How is the entropic discriminant related to the Gauss curve of the central curve?}
The central curve of $A$ is essentially the linear section $X \cap L'$, a smooth
curve of degree $p$ and genus $g$, and its Gauss curve has degree 
${\rm Hdeg}(X) = 2p+2g-2$, by the generalized Pl\"ucker formula.
The Hurwitz form ${\rm Hu}_X$ has that same degree, and hence so does the
entropic discriminant:
\end{example}

\begin{corollary} \label{cor:BtimesA}
The entropic discriminant of the matrix $A$ equals
$(b_1 b_2 \cdots b_d)^{2 \binom{n}{d-1}(d-n)}$ times 
the Hurwitz form in primal Stiefel coordinates
of $X = {\rm rowspace}(A)^{-1}$
evaluated at the matrix
\begin{equation}
\label{eq:BtimesA} \begin{pmatrix}
 b_1 & -b_0 & 0 & 0 & \cdots & 0  \\
  0  & b_2 & -b_1 & 0 & \cdots &  0 \\
  0 &  0   & b_3 & -b_2 & \cdots &  0 \\
  \vdots & \vdots & & \ddots & \ddots &  \\
  0 & 0 & 0  & \cdots & b_d & \!\! - b_{d-1} \\
  \end{pmatrix} \cdot A.
\end{equation}
We  follow \cite{SSV} in using
coordinates $b_0,b_1,\ldots,b_d$ for right hand side vectors
of the matrix $A$.
\end{corollary}

\begin{proof}
The point $(b_0{:}b_1{:}\cdots{:}b_d) \in \PP^d$
lies in the entropic discriminant of $A$ if and only if the
row space of (\ref{eq:BtimesA}) intersects $X$.
The Pl\"ucker coordinate vector of that row space
can be written as a linear expression in $(b_0,b_1,\ldots,b_d)$. It 
is obtained by removing the common factor
$b_1 b_2 \cdots b_{d-1}$ from the maximal minors. This extraneous factor
has exponent $-{\rm Hdeg}(X)$.
\end{proof}

For example, let $d=2, n=4$, and consider the surface $X \subset \PP^4$ defined by the matrix
$$ A \quad = \quad \begin{pmatrix}
 1 & 0 & 0 & 1 & 1 \\
 0 & 1 & 0 & 1 & 0 \\
 0 & 0 & 1 & 0 & 1 \end{pmatrix}. $$
 Explicitly, $X = V( x_0 x_2-x_0 x_4-x_2 x_4, x_0 x_1-x_0 x_3-x_1 x_3)$,
 and ${\rm Hdeg}(X) = 8$. The Hurwitz form
 (in primal Stiefel coordinates) is a homogeneous polynomial ${\rm Hu}_X$
 with $46958$ terms of degree $16$  in the ten entries
 of a $2 \times 5$-matrix of unknowns. If we substitute (\ref{eq:BtimesA}) 
 into ${\rm Hu}_X$ then we obtain $b_1^8$ times the
 sum of squares listed explicitly
  in \cite[Example 1, page 679]{SSV}.
 
 \smallskip
 
 We close with the remark that
  Corollary \ref{cor:BtimesA} can be extended to also
 answer Question~5  in \cite[\S 7, p.706]{SSV},
 as that pertains to translating the variety $X$ via
 the torus action on $\PP^n$.
The 
``Varchenko discriminant'' sought in that question is obtained  by
 multiplying the matrix (\ref{eq:BtimesA}) on the right with the
 diagonal matrix $ {\rm diag}(c_0,c_1,\ldots,c_n)$
 before substituting into ${\rm Hu}_X$.
  Finally, it should be possible to resolve also Question~2 in \cite[\S 7, p.~705]{SSV},
 via the description of the Hurwitz polytope that is implicit in 
 Example~\ref{ex:RLS}. We leave this to a future project.

\bigskip 
\medskip

\noindent
{\bf Acknowledgments.}
This work was supported by the National Science Foundation (DMS-1419018).
The author is grateful to  Alicia Dickenstein,
Jon Hauenstein, Giorgio Ottaviani
and Raman Sanyal for helpful comments.

\smallskip

\begin{small}

\bigskip

\footnotesize 
\noindent {\bf Author's address:}

\noindent Bernd Sturmfels, Department of Mathematics, University of California, Berkeley, CA 94720-3840, USA,
{\tt bernd@berkeley.edu}

\end{small}

\end{document}